\documentclass[11pt, reqno]{amsart}
\usepackage[T1]{fontenc}
\usepackage{amsfonts}
\usepackage{amsmath}
\usepackage{amsthm}
\usepackage{amssymb}
\usepackage{geometry}
\usepackage{graphicx}
\usepackage{xcolor}
\usepackage{mathtools}
\usepackage[colorlinks=true, linkcolor=blue]{hyperref}
\usepackage{cleveref}
\usepackage[english]{babel}
\usepackage{lineno}

\newtheorem{theorem}{Theorem}

\newtheorem{corollary}[theorem]{Corollary}
\newtheorem{lemma}[theorem]{Lemma}

\newtheorem{proposition}[theorem]{Proposition}

\usepackage{tikz}
\usetikzlibrary{positioning}
\usetikzlibrary{decorations,arrows}
\usetikzlibrary{decorations.markings}
\numberwithin{equation}{section}

\usepackage{rustic}
\usepackage[T1]{fontenc}

\emergencystretch=\maxdimen
\title{On the size of sets avoiding a general structure}
\author{Runze Wang}
\address[]{Department of Mathematical Sciences, University of Memphis, Memphis, TN 38152, USA}
\email{runze.w@hotmail.com}
\thanks{}
\date{\today}

\begin{document}

\sloppy

\begin{abstract}
    Given a finite abelian group $G$ and a subset $S\subseteq G$, we let $N_{G,\ S}$ be the smallest integer $N$ such that for any subset $A\subseteq G$ with $N$ elements, we have $g+S\subseteq A$ for some $g\in G$. Using the probabilistic method, we prove that 
    \begin{align*}
        \frac{|H_G(S)|-1}{|H_G(S)|}|G|+\Biggl\lceil\biggl(\frac{|G|}{|H_G(S)|}\biggr)^{1-|H_G(S)|/|S|}\Biggr\rceil\le N_{G,\ S}\le \biggl\lfloor\frac{|S|-1}{|S|}|G|\biggr\rfloor+1,
    \end{align*}
    where $H_G(S)$ is the stabilizer of $S$.
\end{abstract}

\maketitle
Problems about avoiding structures, especially avoiding arithmetic progressions, are well-known and have been extensively studied. For example, the famous Roth's theorem, which is about avoiding three-term arithmetic progressions, was proved in \cite{Ro} and has been refined in \cite{Bo,HB,Sa,Sz}. In this succinct paper, we take the avoided structure to be a general set.

For a finite abelian group $(G,+)$, an element $g\in G$, and a subset $S\subseteq G$, we define $g+S$ to be $\{g+s:s\in S\}$, and define the \textit{stabilizer} of $S$ to be 
\begin{align*}
    H_G(S)=\{g'\in G:\ g'+S=S\}.
\end{align*}
It is easy to check that $H_G(S)$ is a subgroup of $G$, and $S$ is the union of some cosets of $H_G(S)$.

Given a finite abelian group $G$ and a subset $S\subseteq G$, we let $N_{G,\ S}$ denote the smallest integer $N\ge |S|$ such that for any subset $A\subseteq G$ with $N$ elements, we have $g+S\subseteq A$ for some $g\in G$. Thus, for any $M\le N_{G,\ S}-1$, there exists a subset $B\subseteq G$ with $M$ elements, such that $g+S\nsubseteq B$ for any $g\in G$. Roughly speaking, this means the additive structure of $S$ is avoided in $B$.

Firstly we prove the following bounds on $N_{G,\ S}$, and the lower bound will be improved later.

\begin{theorem}\label{thm1}
    We have
    \begin{align*}
        \frac{|H_G(S)|-1}{|H_G(S)|}|G|+1\le N_{G,\ S}\le \biggl\lfloor\frac{|S|-1}{|S|}|G|\biggr\rfloor+1.
    \end{align*}
\end{theorem}

\begin{proof}
    For the lower bound, we can construct a subset $B\subseteq G$ with $\frac{|H_G(S)|-1}{|H_G(S)|}|G|$ elements by excluding one element from each coset of $H_G(S)$, then we will have $g+S\nsubseteq B$ for any $g\in G$.

    For the upper bound, let us assume for some subset $A\subseteq G$ with $\Bigl\lfloor\frac{|S|-1}{|S|}|G|\Bigr\rfloor+1$ elements, we have $g+S\nsubseteq A$ for any $g\in G$, which means $(g+S)\cap (G\setminus A)\neq \emptyset$ for any $g\in G$. For each $\alpha\in G\setminus A$, we have
    \begin{align*}
        |\{g\in G:\ \alpha\in g+S\}|=|\{g\in G:\ g\in \alpha-S\}|=|\alpha-S|=|S|,
    \end{align*}
    which implies
    \begin{align*}
        |\{g\in G:\ (g+S)\cap (G\setminus A)\neq \emptyset\}|&\le \sum_{\alpha\in G\setminus A} |\{g\in G:\ \alpha\in g+S\}| \\
        &=|G\setminus A||S| \\
        &=\biggl(|G|-\biggl(\biggl\lfloor\frac{|S|-1}{|S|}|G|\biggr\rfloor+1\biggr)\biggr)|S| \\
        &<|G|,
    \end{align*}
    contradicting the assumption that $(g+S)\cap (G\setminus A)\neq \emptyset$ for any $g\in G$. So for any subset $A\subseteq G$ with $\Bigl\lfloor\frac{|S|-1}{|S|}|G|\Bigr\rfloor+1$ elements, we can find $g+S$ in $A$ for some $g\in G$, and thus $N_{G,\ S}\le \Bigl\lfloor\frac{|S|-1}{|S|}|G|\Bigr\rfloor+1$.
\end{proof}

We have a direct corollary.
\begin{corollary} \label{cor}
    If $S$ is a coset of some subgroup of $G$, then 
    \begin{align*}
        N_{G,\ S}=\frac{|S|-1}{|S|}|G|+1.
    \end{align*}
\end{corollary}

\begin{proof}
    If $S$ is a coset of a subgroup, then $|H_G(S)|=|S|$, and the equalities in Theorem \ref{thm1} will be attained.
\end{proof}

Let $T_G(S)$ be a transversal of $G/H_G(S)$, which means $T_G(S)$ contains exactly one element from each coset of $H_G(S)$, so $|T_G(S)|=\frac{|G|}{|H_G(S)|}$. For a subset $A\subseteq G$, it is easy to see that the following two statements are equivalent.
\begin{itemize}
    \item There exists $g\in G$, such that $g+S\subseteq A$.
    \item There exists $g'\in T_G(S)$, such that $g'+S\subseteq A$.
\end{itemize}

Using the probabilistic method, we prove another lower bound on $N_{G,\ S}$. We will use this result as a lemma to prove a better lower bound in Theorem \ref{thm2}, which is our final goal.

\begin{lemma} \label{lemma}
    We have
    \begin{align*}
        N_{G,\ S}\ge |T_G(S)|^{-1/|S|}|G|=|H_G(S)|^{1/|S|}|G|^{1-1/|S|}
    \end{align*}
\end{lemma}

\begin{proof}
    Suppose $N\ge |S|$ is an integer such that for any subset $A\subseteq G$ with $N$ elements, there exists $g\in G$ such that $g+S\subseteq A$. We randomly choose a set $X$ from all $N$-element subsets of $G$, then
    \begin{align*}
        \mathbb{P}(\exists\ g\in G\ s.t.\ g+S\subseteq X)&=\mathbb{P}(\exists\ g'\in T_G(S)\ s.t.\ g'+S\subseteq X)\\
        &\le \sum_{g'\in T_G(S)}\mathbb{P}(g'+S\subseteq X)\\
        &=|T_G(S)|\frac{{|G|-|S|\choose N-|S|}}{{|G|\choose N}}\\
        &\le |T_G(S)|\biggl(\frac{N}{|G|}\biggr)^{|S|}.
    \end{align*}
    
    If $N<|T_G(S)|^{-1/|S|}|G|$, then $\mathbb{P}(\exists\ g\in G\ s.t.\ g+S\subseteq X)<1$, which means there is some $N$-element set $A\subseteq G$ such that $g+S\nsubseteq A$ for any $g\in G$, contradiction. So $N\ge |T_G(S)|^{-1/|S|}|G|$, and thus $N_{G,\ S}\ge |T_G(S)|^{-1/|S|}|G|=|H_G(S)|^{1/|S|}|G|^{1-1/|S|}$.
\end{proof}

Combining the ideas in Theorem \ref{thm1} and Lemma \ref{lemma}, we prove the following result. In the proof, we take $G':=G/H_G(S)$ and $S':=S/H_G(S)$, then $H_{G'}(S')$ will be trivial, and by Lemma \ref{lemma}, we have $N_{G',\ S'}\ge |G'|^{1-1/|S'|}$. And because $N_{G',\ S'}$ is an integer, we know $N_{G',\ S'}\ge \bigl\lceil|G'|^{1-1/|S'|}\bigr\rceil$.

\begin{theorem} \label{thm2}
    We have
    \begin{align}
        N_{G,\ S}\ge \frac{|H_G(S)|-1}{|H_G(S)|}|G|+\Biggl\lceil\biggl(\frac{|G|}{|H_G(S)|}\biggr)^{1-|H_G(S)|/|S|}\Biggr\rceil.
    \end{align}
\end{theorem}

\begin{proof}
    We shall construct a subset $B\subseteq G$ with $\frac{|H_G(S)|-1}{|H_G(S)|}|G|+\biggl\lceil\Bigl(\frac{|G|}{|H_G(S)|}\Bigr)^{1-|H_G(S)|/|S|}\biggr\rceil-1$ elements, and show that $g+S\nsubseteq B$ for any $g\in G$.

    By Lemma \ref{lemma}, we know that there is a subset $B'\subseteq G'$ with $\bigl\lceil|G'|^{1-1/|S'|}\bigr\rceil-1$ elements, such that $g'+S'\nsubseteq B'$ for any $g'\in G'$. We let 
    \begin{align*}
        B_1=\{b\in G:\ b+H_G(S)\in B'\},
    \end{align*}
    so
    \begin{align*}
        |B_1|=\bigl(\bigl\lceil|G'|^{1-1/|S'|}\bigr\rceil-1\bigr)|H_G(S)|.
    \end{align*}
    Then, there are $|G'|-|B'|$ cosets of $H_G(S)$ which are not in $B'$, we denote these cosets by $H_1,\ H_2,\ ...,\ H_{|G'|-|B'|}$. In each $H_i$, we randomly pick an element $h_i$, and let $K_i$ be $H_i\setminus\{h_i\}$. We let $B_2$ be the union of $K_i$, so
    \begin{align*}
        B_2=\bigcup_{i=1}^{|G'|-|B'|}K_i,
    \end{align*}
    and
    \begin{align*}
        |B_2|=(|G'|-|B'|)(|H_G(S)|-1)=\bigl(|G'|-\bigl(\bigl\lceil|G'|^{1-1/|S'|}\bigr\rceil-1\bigr)\bigr)(|H_G(S)|-1).
    \end{align*}
    Now, we take $B$ to be $B_1\cup B_2$, then
    \begin{align*}
        |B|&=|G'|(|H_G(S)|-1)+\bigl\lceil|G'|^{1-1/|S'|}\bigr\rceil-1 \\
        &=\frac{|H_G(S)|-1}{|H_G(S)|}|G|+\biggl\lceil\Bigl(\frac{|G|}{|H_G(S)|}\Bigr)^{1-|H_G(S)|/|S|}\biggr\rceil-1.
    \end{align*}
    And we need to show that $g+S\nsubseteq B$ for any $g\in G$. 
    \begin{itemize}
        \item If for some $g\in G$, we have $g+S\subseteq B_1$, then $g':=g+H_G(S)\in G'$ and $g'+S'\subseteq B'$, contradicting the definition of $B'$.
        \item If for some $g\in G$, we have $g+S\subseteq B$ and $(g+S)\cap B_2\neq \emptyset$, then again we have a contradiction, because $g+S$ is a union of $H_G(S)$ cosets, but $B_2$ is a union of $H_G(S)$ cosets with punched holes.
    \end{itemize}
    So $g+S\nsubseteq B$ for any $g\in G$.
\end{proof}

We need to check the lower bound obtained in Theorem \ref{thm2} is better than the one in Lemma \ref{lemma}. Although this is intuitive, we have a formal verification given by the following proposition, where $g$, $h$, and $s$ play the roles of $|G|$, $|H_G(S)|$, and $|S|$ respectively.

\begin{proposition}
    Let $g,\ h,\ s\ge 1$ be three real numbers with $g\ge h$, then 
    \begin{align*}
        \frac{h-1}{h}g+\Bigl(\frac{g}{h}\Bigr)^{1-h/s}\ge h^{1/s}g^{1-1/s}.
    \end{align*}
\end{proposition}

\begin{proof}
    We can fix $g$ and $h$, and take $s$ as a variable. Note that actually we should have $h\le s\le g$, but for calculation convenience, let us take $1\le s<\infty$. Let $f(s)=\frac{h-1}{h}g+(\frac{g}{h})^{1-h/s}-h^{1/s}g^{1-1/s}$. It turns out $f'(s)\le 0$, so $f(s)$ is decreasing on $[1,\ \infty)$. And if $s$ is taken to be $\infty$, then $f(\infty)=0$. So we always have $f(s)\ge 0$, and thus $\frac{h-1}{h}g+\bigl(\frac{g}{h}\bigr)^{1-h/s}\ge h^{1/s}g^{1-1/s}$.
\end{proof}

Note that if $S$ is a coset of some subgroup of $G$, then the lower bound in Theorem \ref{thm2} is the same as the one in Theorem \ref{thm1}; if $|H_G(S)|=1$ or $H_G(S)=S=G$, then the lower bound in Theorem \ref{thm2} is the same as the one in Lemma \ref{lemma}.

We finish this short paper with an example.

Let us take $G$ to be $C_{2024}$, the cyclic group of order $2024=2^3\cdot 11\cdot 23$, and take $S$ to be the union of $n$ cosets of the subgroup of order eight. So $|S|=8n$, and if we restrict $n$ to be in $[1,\ 10]$, then $|H_G(S)|$ must be eight. Then by Theorem \ref{thm1}, Corollary \ref{cor}, and Theorem \ref{thm2}, we have
    \[ N_{G,\ S}\begin{cases} 
          =1772 & if\ n=1, \\
          \in [1787,\ 1898] & if\ n=2, \\
          \in [1812,\ 1940] & if\ n=3, \\
          \in [1835,\ 1961] & if\ n=4, \\
          \in [1855,\ 1974] & if\ n=5, \\
          \in [1872,\ 1982] & if\ n=6, \\
          \in [1886,\ 1988] & if\ n=7, \\
          \in [1898,\ 1993] & if\ n=8, \\
          \in [1908,\ 1996] & if\ n=9, \\
          \in [1917,\ 1999] & if\ n=10.
       \end{cases}
    \]

\end{document}